\documentclass[a4paper,11pt]{article}
\usepackage[utf8]{inputenc}
\usepackage[all]{xy}

\usepackage{amsmath,amssymb,amsthm,enumerate}

\usepackage[pagewise,mathlines]{lineno}

\usepackage{xcolor}

\makeatletter
\def\makeLineNumberLeft{%
 \linenumberfont\llap{\hb@xt@\linenumberwidth{\LineNumber\hss}\hskip\linenumbersep}
 \hskip\columnwidth
 \rlap{\hskip\linenumbersep\hb@xt@\linenumberwidth{\hss\LineNumber}}\hss}
\leftlinenumbers
\makeatother

\usepackage[hidelinks]{hyperref}

\usepackage{aliascnt}

\newcommand{\mynewtheorem}[2]{
 \newaliascnt{#1}{dummy}
 \newtheorem{#1}[#1]{#2}
 \aliascntresetthe{#1}
 \expandafter\def\csname #1autorefname\endcsname{#2}
}

\newcommand{\mynewnumbered}[2]{
 \newaliascnt{#1}{numbered}
 \newtheorem{#1}[#1]{#2}
 \aliascntresetthe{#1}
 \expandafter\def\csname #1autorefname\endcsname{#2}
}

\theoremstyle{plain}
 \mynewtheorem{theorem}{Theorem}
 \mynewtheorem{proposition}{Proposition}
 \mynewtheorem{lemma}{Lemma}
 \mynewtheorem{corollary}{Corollary}
\theoremstyle{definition}
 \mynewnumbered{definition}{Definition}
 \mynewnumbered{remark}{Remark}
 \mynewnumbered{example}{Example}
 
\newcommand{\K}{\mathbb K}
\newcommand{\Q}{\mathbb Q}
\newcommand{\C}{\mathbb C}
\newcommand{\Z}{\mathbb Z}
\newcommand{\GL}{\mathrm{GL}}
\newcommand{\PGL}{\mathrm{PGL}}

\newcommand{\Hom}{\mathrm{Hom}}

\newcommand{\Exel}[1]{\mathcal{S}(#1)}
\newcommand{\pSym}{\mathcal{I}}

\newcommand{\gen}[1]{\langle #1 \rangle}
\newcommand{\geni}[1]{\langle #1 \rangle_\mathcal{S}}
\newcommand{\inv}[1]{{#1^{-1}}}
\newcommand{\rk}{\mathrm{rk}}

\begin{document}
\title{Schur's theory for partial projective representations}
\author{Mikhailo Dokuchaev and Nicola Sambonet\thanks{
The first author was partially supported by Fapesp of Brazil and by CNPq of Brazil.
The second author was supported by Fapesp of Brazil.
Address: Instituto de Matem\' atica e Estat\' istica da Universidade de S\~ao Paulo, Rua do Mat\~ao 1010, 05508-090 S\~ao Paulo, Brazil.
Email: dokucha@gmail.com and nsambonet@gmail.com
}}

\maketitle

\begin{abstract}
 This article focuses on those aspects about partial actions of groups which are related to Schur's theory on projective representations.
 It provides an exhaustive description of the partial Schur multiplier, and this result is achieved by introducing the concept of a second partial cohomology group relative to an ideal, together with an appropriate analogue of a central extension. In addition, the new framework is proved to be consistent with the earlier notion of cohomology over partial modules.
\end{abstract}

\section{Introduction}
It is becoming relatively common to encounter the word \emph{partial} as a prefix for some classical terminology in algebra.
This phenomenon is due to the increasing awareness about the algebraic relevance of partial symmetries which explicitly or implicitly appear in various local-global aspects of  well-known mathematical topics.
It was indeed the theory of $C^\ast$-algebras which motivated the introduction of the notions of partial actions and partial representations of groups \cite{MR1469405}.
In this regard, the most prominent application of the partial actions is their involvement in the construction of a crossed product, which encloses relevant classes of algebras  \cite{MR2799098}. Partial representations are intimately related to partial actions in various ways, in particular, they play a fundamental role in providing partial crossed pro\-duct structures on concrete algebras. This can be seen applied to a series of important algebras, among the most recent examples being $C^*$-algebras related  to dynamical systems of type $(m,n)$  \cite{MR3103084}, the Carlsen-Matsumoto algebras of subshifts  \cite{MR3639522} and  algebras related to separated graphs  \cite{MR3144248}. In the latter paper, a partial-global passage was remarkably employed  to a problem on paradoxical decompositions.

The general concept of a twisted partial group action, used to define more general crossed products, involves a kind of a $2$-cocycle (see \cite{MR2450727,MR1425329}), raising the  desire to fit it into some cohomology theory.
In order to obtain some testing material for the development of such a theory, partial projective group representations were introduced and studied in \cite{MR2559695,MR2835207,MR3085031,MR3447601}.
The notion of a partial multiplier naturally appears in analogy with the classical Schur's theory of projective representations.
Since the algebraic structures involved in  partial representations are not only groups but also inverse semigroups, the new theory diverges from the original one and the multiplier results in a semilattice of abelian groups, called components.
That is briefly to say, a collection composed of abelian groups and parametrized by a meet semilattice, where the operations of group-multiplications are compatible with the partial order. 
Subsequently, a cohomology theory based on partial actions was developed in  \cite{MR3312299,DKh2,DKh3}.
Notice that partial cohomology turned out to be  useful to deal with the ideal separation property  of  (global) reduced $C^*$-crossed products   in \cite{KennedySchafhauser}, where partial $2$-cocyles  appear 
as factor sets (twists) of  partial projective representations (called twisted partial representations in \cite{KennedySchafhauser}) naturally related to (global) $C^*$-dynamical systems.

With respect to  the partial Schur multiplier, it was shown in \cite[Theorem 2.14]{MR3312299} that each component  is a union of cohomology groups with values in non-necessarily trivial partial modules. Furthermore, in the case of an algebraically closed field  $\K$, it was proved in \cite[Theorem 5.9]{MR3085031} that each component  of the partial Schur multiplier is an epimorphic image of a direct product of copies of  $\K^\times.$ However, computations of the  partial Schur multiplier of concrete groups show that each component is, in fact,  isomorphic to a  direct product of copies of  $\K^\times$ (see \cite{MR3492987,MR3085031,MR3169719,Pinedo2014,MR3447601,MR3562686}), raising the conjecture that this should be true for all  groups.

The present article is inspired by the aforementioned conjecture, of which it provides an affirmative answer for finite groups   (\autoref{Cor:partial multiplier}), and it actuates an essential refinement for our understanding  of partial projective group representations.
To this aim, the theory is reformed to encompass a broader setting, specifically to allow arbitrary coefficients and to establish a partial analogue for Schur's theory of central extensions (in particular, see \autoref{Thm:partial cohomology and extensions}).
Influenced by known  properties of partial factor sets, the concepts of  a \emph{pre-cohomology group}, of  a \emph{second partial cohomology group relative to an ideal}, and of a \emph{second partial cohomology semilattice of groups} with values in an arbitrary abelian group are introduced.
Then the partial Schur multiplier is shown  to coincide with  the second  partial cohomology semilattice of groups with values in   $\K^\times$.
This new level of abstraction conveys standard techniques into the scene, which promptly produce new theorems.

Most remarkably, in the finite case the pre-cohomology groups are completely determined by minor information about the underlying group (\autoref{Thm:precohomology group}).
Also, taking the ideals into accounts, a general statement is established for the second partial cohomology groups (\autoref{Thm:precohomology group and ideals}).
Furthermore, considering  appropriate partial $G$-modules,  one is able to conclude, in particular, that each component of the partial Schur multiplier is isomorphic to a single partial cohomology group (see \autoref{Thm:coherence with the previous notion}) in the sense of    \cite{MR3312299},  bringing thus the partial Schur theory closer to the classical one.

The text is structured as follows:
in \autoref{Section:prerequisites} some fundamental concepts in the theory of partial actions of groups and in Schur's theory will be recalled;
in \autoref{Section:analogies}  new concepts will be introduced, and the partial analogue for Schur's theory with arbitrary coefficients will be developed;
in \autoref{Section:pre-cohomology} and, respectively, in \autoref{Section:partial cohomology} the pre-cohomology groups and the second partial cohomology semilattice of groups will be studied in details;
lastly, in \autoref{Section:coherence} the consistency with the previous notion of partial cohomology will be proved.

\section{Prerequisites}\label{Section:prerequisites}
\subsection{The universal inverse semigroup of a group}
Inverse semigroups are the algebraic structures which capture the idea of partial symmetry \cite{MR0466355}.
In an \emph{inverse semigroup} $S$ the system of equations
\begin{equation}\label{Eq:inverse semigroup}
 xyx=x\ \wedge \ yxy=y
\end{equation}
admits a unique solution for $y$ in terms of $x$, commonly denoted by $y=\inv{x}$.
This is easily recognized to be a generalization of groups, still it incorporates the semilattices, which are the semigroups consisting of commuting idempotents.
A \emph{partial bijection} over a set $X$ is any bijection between subsets of $X$.
Multiplication of partial bijection is defined as follows:
if $\alpha:A\to A^\alpha$ and $\beta:B\to B^\beta$, then $\beta \circ \alpha  $ is their composition on the largest subset where it makes sense, i.~e.~$\inv{\alpha}(A^\alpha\cap B)$.
Necessarily, one has to admit the existence of the map $\infty:\emptyset\to\emptyset$, and this serves as the zero element.
As the reader has noticed, we convene to denote with $\infty$ the zero element of a semigroup, if this ever exists, coherently with the additive notation for abelian groups.
The set of partial bijections over $X$ turns out to be an inverse semigroup in the obvious way, that is by inversion of functions, which is named the \emph{symmetric inverse semigroup} $\pSym(X)$.

We recall that, given an ideal $I$ of a semigroup $S$, the \emph{Rees congruence} is defined by setting $x\sim y$ when either $x=y$ or both $x$ and $y$ belong to $I$. The corresponding quotient is denoted by $S/I$. In general, if $\vartheta:S\to T$ is a homomorphism of semigroups and $T$ admits the zero $\infty$, then $\inv{\vartheta}(\infty)$ is an ideal of $S$.

Inverse semigroups apply to the theory of groups through the notion of partial actions and partial homomorphisms of groups \cite{MR1469405}.
A \emph{partial action} is a map $\psi:G\to\pSym(X)$ satisfying the compatibility conditions that $\psi(1)=\mathrm{id}_X$ and the composition $\psi(g)\circ\psi(h)$ coincides with the restriction of $\psi(gh)$.
In turn, this notion results to be a special case of the following.
A \emph{partial homomorphism} is a map $\psi$ from a group $G$ into a semigroup $M$ satisfying the conditions
\begin{align}
&\psi(1)\cdot\psi(g)=\psi(g)=\psi(g)\cdot\psi(1)\label{Eq:partial homomorphism 1}\\
&\psi(\inv{g})\cdot\psi(g)\cdot\psi(h)=\psi(\inv{g})\cdot\psi(gh)\label{Eq:partial homomorphism 2}\\
&\psi(g)\cdot\psi(h)\cdot\psi(\inv{h})=\psi(gh)\cdot\psi(\inv{h})\label{Eq:partial homomorphism 3}
\end{align}
for any pair of elements $g$ and $h$ of $G$.
These relations, which openly emerge from the compatibility requirement of partial action, are used to introduce the \emph{universal inverse semigroup} $\Exel{G}$ \emph{associated with} $G$ (see \cite{MR1469405}). Precisely, $\Exel{G}$ is the abstract semigroup generated by symbols $\Pi(g)$ subject to relations analogue to \eqref{Eq:partial homomorphism 1}-\eqref{Eq:partial homomorphism 3}, so that the map $\Pi:G\to\Exel{G}$ is an injective partial homomorphism, and this construction is universal in sense that
\begin{equation}\label{Eq:Exel universal}
 \xymatrix{&\Exel{G}\ar@{.>}[d]^{\bar\psi}\\G\ar[ur]^\Pi\ar[r]^(.55)\psi&M}
\end{equation}
every partial homomorphism $\psi$ extends to a semigroup homomorphism $\bar\psi$ uniquely by means of $\Pi$.

In fact, $\Exel{G}$ is isomorphic to the Birget-Rhodes expansion of $G$  \cite{MR2041539} (see also \cite{MR745358,MR989900}).
Thus, a generic element of $\Exel{G}$ is identified with a pair $(R,g)$ where $R$ is a finite subset of $G$ containing $1$ and $g$,
whereas multiplication obeys to the rule
\[(R,g)\cdot(S,h)=(R\cup gS,gh)\ .\]
The result is an inverse semigroup since the system of equations \eqref{Eq:inverse semigroup} admits the unique solution
$\inv{(R,g)}=(\inv{g}R,\inv{g})$, and the universal partial homomorphism $\Pi$ associates an element $g$ of the group with $\Pi(g)=(\{1,g\},g)$.
It is convenient to extend the definition of $\Pi$ to a higher number of elements as $\Pi(g_1,\ldots,g_k)=\Pi(g_1)\cdots\Pi(g_k)$.
In particular, considering two elements yields
\begin{equation}\label{Eq:product of two generators}
 \Pi(g,h)=(\{1,g\},g)\cdot(\{1,h\},h)=(\{1,g,gh\},gh)\ ,
\end{equation}
whereas considering three elements yields
\begin{equation}\label{Eq:product of three generators}
 \Pi(x,y,z)=(\{1,x\},x)\cdot(\{1,y\},y)\cdot(\{1,z\},z)=(\{1,x,xy,xyz\},xyz)\ .
\end{equation}
One may notice how these relations confirm that $\Pi$ satisfies \eqref{Eq:partial homomorphism 1}-\eqref{Eq:partial homomorphism 3}.

It is important to describe the ideals of $\Exel{G}$ since these determine the vanishing of a partial homomorphism.
Precisely, in the notation of \eqref{Eq:Exel universal} the equality $\psi(g_1)\cdots\psi(g_k)=\infty$ occurs if and only if $\Pi(g_1,\ldots,g_k)$ belongs to $I=\inv{\bar\psi}(\infty)$.
To begin with, the principal ideals in $\Exel{G}$, which we will denote by $\geni{\cdot}$, can be characterized as follows \cite[Corollary 2]{MR2835207}:
\begin{equation}\label{Eq:principal ideals}
 \geni{(R,g)}=\geni{(S,h)}\iff(S,h)=(\inv{x}R,\inv{x}y)\mbox{ for some } x,y\in R\ .
\end{equation}
For our purpose, one may just consider the following relations which are direct consequences of the definition of a partial representation:
\begin{equation}\label{Eq:generation of ideals inverses}
 \geni{\Pi(g,\inv{g})}=\geni{\Pi(g)}=\geni{\Pi(\inv{g})}=\geni{\Pi(\inv{g},g)}\ ,
\end{equation}
and more generally, for any $1\leq j\leq r$ and $1\leq k\leq r$,
\begin{equation}\label{Eq:ideals and associativity}
	\begin{array}{rl}
	 \geni{\Pi(g_{1},\ldots,g_{r})}&=\geni{\Pi(\inv{(g_1\cdots g_j)},g_1,\ldots,g_r)}\\
   &=\geni{\Pi(g_1,\ldots,g_r,\inv{(g_{k}\cdots g_r)})}\\
   &=\geni{\Pi(\inv{(g_1\cdots g_j)},g_1,\ldots,g_r,\inv{(g_{k}\cdots g_r)})}\ .
	\end{array}
\end{equation}
In addition, with respect to \eqref{Eq:product of two generators}, taking $y=1$ and either $x=g$ or $x=gh$ in the above characterization \eqref{Eq:principal ideals} yields
\begin{equation}\label{Eq:generation of ideals}
 \geni{\Pi(g,h)}=\geni{\Pi(h,\inv{h}\inv{g})}=\geni{\Pi(\inv{h},\inv{g})}\ .
\end{equation}
These identities prove that, given a partial homomorphism $\psi$, the subset of $G\times G$ consisting of those pairs $(g,h)$ for which $\psi(g)\cdot\psi(h)=\infty$ is invariant under the action of the symmetric group $S_3$ generated by the transformations
\begin{equation}\label{Eq:S3 action}
 (g,h)\mapsto(h,\inv{h}\inv{g})\ ,\ (g,h)\mapsto(\inv{h},\inv{g})\ .
\end{equation}
This action, which has been discovered in \cite{MR2559695}, will be central in the present essay.

Another family of ideals can be described by the simple observation that, for any pair of elements $(R,g)$ and $(S,h)$ of $\Exel{G}$, since $|R\cup gS|\geq|S|$ and $|R\cup gS|\geq|R|$, left and right multiplication are ``non-decreasing operations".
Consequently, for any positive integer $k$, the set
\[N_k=\{(R,g)\in\Exel{G}\ |\ |R|\geq k+1\}\]
is an ideal of $\Exel{G}$.

In the present essay the crucial case is when $k=3$. Indeed, in the notation of \eqref{Eq:Exel universal}, for any partial homomrphisms $\psi$ considered hereby it will be assumed that the ideal $I=\inv{\bar\psi}(\infty)$ is proper in $\Exel{G}$ and it contains $N_3$.
This assumption is motivated by the theory on partial projective representations where the semilattice
\begin{equation}\label{Eq:Lambda}
 \Lambda=(\Lambda,\cup)\ \ ,\ \ \Lambda=\{I\ |\ N_3\leq I\triangleleft\Exel{G}\}
\end{equation}
plays a prominent role (see \cite{MR3085031}).

Taking this into account, suppose that $\Pi(x,y,z)$ is not contained in $I$, for $I=\inv{\bar\psi}(\infty)$ lying in $\Lambda$.
Accordingly with \eqref{Eq:product of three generators}, one has $|\{1,x,xy,xyz\}|\leq 3$, which can be rewritten as
\begin{equation}\label{Eq:1 notin xyzxyyzxyz}
 1\in\{x,y,z,xy,yz,xyz\}\ .
\end{equation}
Therefore, the failure on this condition implies $\psi(x)\cdot\psi(y)\cdot\psi(z)=\infty$.

\subsection{A sketch of classical Schur's theory}
The ordinary multiplier is the fundamental object in the theory of projective representations and central extensions \cite{MR0460423,MR1200015,Schur1904}.
Hereby, few of the very elementary concepts are mentioned in order to fix the notation and facilitate the description of their partial analogue.

A \emph{central extension} of a group $M$ is a group $E$ endowed with a surjective homomorphism $\pi:E\to M$ such that $A=\ker\pi$ is contained in the center $Z(P)$.
This situation is summarized in the diagram
\begin{equation}\label{Eq:central extension}
 \xymatrix{A\ar[r]^\iota&E\ar[r]^\pi&M}
\end{equation}
The interest is to consider such a central extension together with a group $G$ and a homomorphism $\psi:G\to M$.
The archetypes are:
\begin{enumerate}[i.]
 \item When $M=G$ and $\psi$ is the identity, then $E$ is a \emph{central extension} of $G$.
 \item When $E=\GL(n,\C)$ and $M=\PGL(n,\C)$ for some positive integer $n$, then $\psi$ is a \emph{projective representation} of $G$.
\end{enumerate}
Hence, the choice of a section $\varphi:G\to E$ completes the commutative diagram:
\begin{equation}\label{Eq:diagram}
\xymatrix{
&&G\ar[d]^\psi\ar@{.>}[dl]_\varphi\\
A\ar[r]^\iota&E\ar[r]^\pi&M
}\end{equation}
One can not expect that the map $\varphi$ is a homomorphism.
However, the failure for this to happen is encoded in the equation
\begin{equation}\label{Eq:varphi and sigma}
 \varphi(g)\cdot\varphi(h)=\varphi(gh)\cdot\sigma(g,h)
\end{equation}
where $\sigma$ is a function in $A^{G\times G}$.
Clearly, not any function is admissible: adopting the additive notation for group multiplication in $A$, associativity of $E$ proves that $\sigma$ satisfies the equality
\[\sigma(x,y)+\sigma(xy,z)=\sigma(x,yz)+\sigma(y,z)\]
for any $x,y$ and $z$ in $G$. In connection with cohomology this fact is restated as follows.
The coboundary homomorphism
\[\delta^2:A^{G^{\times 2}}\to A^{G^{\times 3}}\]
is defined by means of the relation
\begin{equation}\label{Eq:defining cocycles}
 \delta^2\sigma(x,y,z)=\sigma(y,z)-\sigma(xy,z)+\sigma(x,yz)-\sigma(x,y)\ ,
\end{equation}
and the \emph{group of cocycles} is the kernel of this homomorphism
\[Z^2(G,A)=\{\sigma\in A^{G\times G}\ |\ \delta^2\sigma=0\}\ .\]
Thus, the function $\sigma$ determined by \eqref{Eq:varphi and sigma} is an element of $Z^2(G,A)$. 
Luckily enough the process is reversible and any $\sigma\in Z^2(G,A)$ defines a central extension \eqref{Eq:central extension} of the group $G$ as follows. The underlying set is $A\times G$ and multiplication is given by
\begin{equation}\label{Eq:central extension from cocycle}
 (\lambda,g)\cdot(\mu,h)=(\lambda+\mu+\sigma(g,h),gh)\ .
\end{equation}
Therefore, there is a correspondence between the central extensions and the cocycles of $G$ which, however, depends on the choice of section.
Nonetheless, any change of section, as it can be seen easily, corresponds to addition of a coboundary
\begin{equation}\label{Eq:defining coboundaries}
 \delta^1\zeta(g,h)=\zeta(h)-\zeta(gh)+\zeta(g)\ .
\end{equation}
Hence, it is of interest to introduce the group
\[B^2(G,A)=\{\delta^1\zeta\ |\ \zeta\in A^G\}\ ,\]
and to consider the \emph{second cohomology group}
\[H^2(G,A)=Z^2(G,A)/B^2(G,A)\ .\]
The terminology \emph{Schur multiplier} applies to the case when $G$ is a finite group and $A=\C^\times$,
which answers to the original problem of classifying the projective representations.
For general coefficients, the facts here discussed provide a one-to-one correspondence between the isomorphism classes of the central extensions of $G$ by $A$ and the elements of $H^2(G,A)$

Of course, the above definition is a special case in the theory of group cohomology (see \cite{MR672956}).
Indeed, in the present manuscript it is always assumed that the (global) group actions over the coefficient modules are trivial.
In view of this, the \emph{first cohomology group} simply is
\[H^1(G,A)=\ker\delta^1=\Hom(G,A)\ .\]

\section{The partial analogue notions}\label{Section:analogies}
Also the partial counterpart of the theory originates with the study of partial projective representations in $\K^\times$-cancellative monoids \cite{MR2559695,MR2835207,MR3085031} (previously named $\K$-cancellative).
However, the most general situation can be extended to semigroups and partial homomorphisms.

Firstly, referring to the diagram \eqref{Eq:central extension}, in the current framework $E$ will be taken as an $A$-cancellative central extension of a semigroup $M$ by an abelian group $A$, in the following sense.
A \emph{central extension} is a semigroup $E$ which contains $A$ as a central subgroup, and $M$ is the quotient semigroup $E/A$ with respect to the following congruence: $x\sim y$ whenever there exist $a$ and $b$ in $A$ such that $ax=by$.
Given a central extension $E$ endowed with zero $\infty$, this extension is called $A$-\emph{cancellative} if, for every $a,b$ in $A$ and every $x$ in $E\setminus\{\infty\}$, the equality $ax=bx$ implies $a=b$.
Observe that, this condition yields that, for every $a$ in $A$, the equation $ax=\infty$ has no solution for $x$ in $E$ other than $x=\infty$.

In second place, the due modification in the diagram \eqref{Eq:diagram} is given by  an $A$-cancellative central  extension of $M$ together with a group $G$ and a partial homomorphism $\psi:G\to M$.
Notice that, when $\infty\in E$, then $M$ has at least two elements.
In this case, it is convenient to assume that $\psi(1)\neq\infty$ since, otherwise, $\psi(g)=\infty$ for all $g$.
With respect to \eqref{Eq:partial homomorphism 2}-\eqref{Eq:partial homomorphism 3} and \eqref{Eq:varphi and sigma}, \emph{this time} the failure for the section $\varphi$ to be a \emph{partial} homomorphism can be shown to be encoded in the equations
\begin{align}
 &\varphi(\inv{g})\cdot\varphi(g)\cdot\varphi(h)=\varphi(\inv{g})\cdot\varphi(gh)\cdot\sigma(g,h)\label{Eq:defining precocycles}\\
 &\varphi(g)\cdot\varphi(h)\cdot\varphi(\inv{h})=\varphi(gh)\cdot\varphi(\inv{h})\cdot\sigma(g,h)\label{Eq:defining precocycles right}
\end{align}
where $\sigma(g,h)$ is an element of $A$.
Actually, it is convenient to regard $\sigma(g,h)$ as an element of the semigroup $A_\infty=A\cup\{\infty\}$, obtained adjoining the zero $\infty$ to the abelian group $A$,  and set $\sigma(g,h)=\infty$ whenever the above equations reduce to the triviality $\infty=\infty$.

Similarly with the case of a partial projective representation over a field, the fact that the same function $\sigma$ appears in both \eqref{Eq:defining precocycles} and \eqref{Eq:defining precocycles right} is not obvious, but the arguments given in the proof of \cite[Theorem 3]{MR2559695} can be easily adapted to this case.
Moreover, also in the current situation the map $\sigma$ has not to satisfy the cocycle identity \cite[page 260]{MR2559695}.
However, by means of \eqref{Eq:ideals and associativity}, associativity of $E$ applied to the relation $\varphi(\inv{x})\cdot\varphi(x)\cdot\varphi(y)\cdot\varphi(z)\cdot\varphi(\inv{z})$
shows that $\delta^2\sigma(x,y,z)=0$ whenever $\Pi(x,y,z)$ does not belong to the ideal $I=\inv{\psi}(\infty)$.
Accordingly to \eqref{Eq:1 notin xyzxyyzxyz}, it is of interest to introduce the following:
\begin{definition}\label{Def:pre-cohomology}
 A \emph{pre-cocycle} of $G$ with values in $A$ is a function  $\sigma$ in $A^{G\times G}$ satisfying the weak cocycle condition:
 \[1\in\{x,y,z,xy,yz,xyz\}\Rightarrow\delta^2\sigma(x,y,z)=0\ .\]
 The group of pre-cocycles is denoted by $pZ^2(G,A)$, and
 \[pH^2(G,A)=pZ^2(G,A)/B^2(G,A)\]
 is the \emph{pre-cohomology group of} $G$.
\end{definition}
The zero $\infty$, which does not appear in the above definition, is reintroduced as follows.
Any fixed ideal $I$ in the semilattice $\Lambda$, introduced in \eqref{Eq:Lambda}, defines an additive characteristic function $\epsilon_I:G\times G\to A_\infty$ obeying to
 \[\epsilon_I(g,h)=\left\{\begin{array}{cl}\infty&\Pi(g,h)\in I\\0&\mbox{otherwise.}\end{array}\right.\]
In these terms, the map $A^{G\times G}\to A_{\infty}^{G\times G}$ associating $\alpha\mapsto\alpha+\epsilon_I$ is a homomorphism of semigroups.
\begin{definition}\label{Def:partial cohomology group}
In the current notation, define
\[Z^2(G,I;A)=pZ^2(G,A)+\epsilon_I\ \ ,\ \ B^2(G,I;A)=B^2(G,A)+\epsilon_I\ .\]
Then, the quotient \[H^2(G,I;A)=Z^2(G,I;A)/B^2(G,I;A)\ ,\]
is the \emph{second partial cohomology group relative to} $I$ with coefficients in $A$.
\end{definition}
We shall see below (\autoref{Thm:coherence with the previous notion}) that this definition is coherent with the notion of partial cohomology groups given in \cite{MR3312299}. Moreover, the fact \[Z^2(G,N_3;A)=pZ^2(G,A)\] justifies that only ideals containing $N_3$ are considered.

Clearly, the identity element of $H^2(G,I;A)$ is the class $[\epsilon_I]$.
In addition, in $A_\infty^{G\times G}$ one has $\epsilon_I+\epsilon_J=\epsilon_{I\cup J}$. Thus, by allowing external sums, that is $[\sigma+\epsilon_I]+[\tau+\epsilon_J]=[\sigma+\tau+\epsilon_{I\cup J}]$,
it is possible to join all of these groups as follows:
\begin{definition}\label{Def:partial semilattice}
The \emph{second partial cohomology semilattice of groups of} $G$ is
\[H^2(G,\Lambda;A)=\coprod_{I\in\Lambda}H^2(G,I;A)\]
\end{definition}
The semigroup of idempotents in $H^2(G,\Lambda;A)$ consists of the classes $[\epsilon_I]$ as $I$ varies in $\Lambda$, and it is canonically isomorphic with $(\Lambda,\cup)$.
Furthermore, the definition of $Z^2(G,-;A)$ and $B^2(G,-;A)$ immediately yields:
\begin{proposition}
 For any pair of ideals $I$ and $J$ in $\Lambda$, the homomorphism $H^2(G,J;A)\to H^2(G,I\cup J;A)$ mapping $\mu\mapsto\mu+[\epsilon_I]$ is surjective.
\end{proposition}
From the above discussion, it follows that a map $\sigma$ defined by \eqref{Eq:defining precocycles}-\eqref{Eq:defining precocycles right} is an element of $Z^2(G,I;A)$ for $I=\inv{\psi(\infty)}$.
The following construction, which is analogue to \eqref{Eq:central extension from cocycle}, provides the converse of this statement. Namely, for any ideal $I$ in $\Lambda$, every element of $Z^2(G,I;A)$ can be obtained as a map $\sigma$ arising from an $A$-cancellative central extension together with a partial homomorphism, as in \eqref{Eq:defining precocycles}-\eqref{Eq:defining precocycles right}.
\begin{lemma}\label{Lem:extension of E3 by a precocycle}
Given $\sigma$ in $Z^2(G,I;A)$, then the set $E=\left(A\times\Exel{G}\setminus I\right)\cup\{\infty\}$ endowed with multiplication
\[(\lambda,R,x)\cdot(\mu,S,y)
=\left\{\begin{array}{cl}
        (\lambda+\mu+\sigma(x,y),R\cup xS,xy)&\mbox{ if }(R\cup xS,xy)\notin I\\
        \infty&\mbox{ otherwise}
        \end{array}\right.\]
 results in an $A$-cancellative central extension of $\Exel{G}/I$.
 Moreover, $E$ is a monoid with identity element $(-\sigma(1,1),\Pi(1))$, and the natural partial homomorphism $G\to\Exel{G}/I$ admits the section 
\[\varphi:G\to E\ ,\ g\mapsto\left\{\begin{array}{cl}
        (-\sigma(1,1),\Pi(g))&\mbox{ if }\Pi(g)\notin I\\
        \infty&\mbox{ otherwise}
        \end{array}\right.\]
which is associated with $\sigma$. 
\end{lemma}
\begin{proof}
To prove the associativity of $E$, given $X=(R,x)$, $Y=(S,y)$ and $Z=(T,z)$ in $\Exel{G}\setminus I$ together with $\lambda$, $\mu$ and $\nu$ in $A$, one considers the two ways for computing the product $(\lambda,X)\cdot(\mu,Y)\cdot(\nu,Z)$.
There is no loss of generality assuming that $XYZ$ does not belong to $I$, since otherwise the product results $\infty$.
In particular, associativity corresponds to the cocycle identity $\delta^2\sigma(x,y,z)=0$ under this condition.
Write $XYZ=(R\cup xS\cup xyT,xyz)$.
Since $\{1,x,xy,xyz\}$ is contained in $R\cup xS\cup xyT$ and $N_3$ is contained in $I$, 
it follows that $|\{1,x,xy,xyz\}|\leq 3$, that is $1\in\{x,y,z,xy,yz,xyz\}$.
This proves the associativity of $E$ since, by hypothesis, $\sigma=\sigma'+\epsilon_I$ for some $\sigma'\in pZ^2(G,A)$.
Finally, denoting by $\psi$ the canonical partial homomorphism $G\to\Exel{G}/I$, sending $g$ to $\Pi(g)$ whenever the latter is not in $I$, and to $\infty$ otherwise, and taking the section $\varphi:G\to E$ which accordingly maps $g$ either to $(-\sigma(1,1),\psi(g))$ or to $\infty$, it is evident that the associated map is $\sigma=\sigma'+\epsilon_I$.
\end{proof}
It is routine to check that the isomorphism class of the extension $E$ only depends on the partial cohomology class of $\sigma$ in $H^2(G,I;A)$, consequently:
\begin{theorem}\label{Thm:partial cohomology and extensions}
 For any ideal $I$ in $\Lambda$, there is a one-to-one correspondence between the equivalence classes of $A$-cancellative central extensions of $\Exel{G}/I$ and the elements of $H^2(G,I;A)$.
\end{theorem}
In particular, in the case $A=\K^\times$, any partial cocycle in $Z^2(G,\Lambda;\K^\times)$ defines a $\K^\times$-cancellative monoid together with a partial projective representation, and vice-versa. Therefore, the partial multiplier $pM(G)$ introduced in \cite{MR2559695} turns out to be a partial cohomology semilattice of groups:
\begin{proposition}\label{Prop:partial multiplier and precohomology}
 Let $G$ be a group and $\K$ a field. Then the partial multiplier $pM(G)$ over $\K$ coincides with $H^2(G,\Lambda;\K^\times)$.
\end{proposition}

\section{Pre-cohomology}\label{Section:pre-cohomology}
The pre-cocycles can be characterized in several different ways:
\begin{lemma}\label{Lem:characterization of precocycles}
 Let $\sigma\in A^{G\times G}$. The following properies are equivalent:
 \begin{enumerate}[i.]
  \item $\sigma$ is a pre-cocycle.
  \item $1\in\{xy,xyz\}\Rightarrow\delta^2\sigma(x,y,z)=0$
  \item For every $g,h\in G$ the following holds:
  \begin{enumerate}[a.]
   \item $\sigma(h,\inv{h}\inv{g})=\sigma(g,h)+\sigma(gh,\inv{h}\inv{g})-\sigma(g,\inv{g})$
   \item $\sigma(\inv{h},\inv{g})=-\sigma(h,\inv{h}\inv{g})+\sigma(h,\inv{h})+\sigma(1,1)$
  \end{enumerate}
 \end{enumerate}
Moreover, if the above properties are satisfied, then:
 \begin{enumerate}[i.]
  \item[]
  \begin{enumerate}[a.]
   \item[c.] $\sigma(g,1)=\sigma(1,1)=\sigma(1,g)$
   \item[d.] $\sigma(g,\inv{g})=\sigma(\inv{g},g)$
   \item[e.] $\sigma{ (\inv{h}\inv{g},g)}=\sigma(g,h)+\sigma{ (gh,\inv{h}\inv{g})}-\sigma{ (h,\inv{h})}$
   \item[f.] $\sigma{ (gh,\inv{h})}=-\sigma(g,h)+\sigma{ (h,\inv{h})}+\sigma{ (1,1)}$
   \item[g.] $\sigma{ (\inv{g},gh)}=-\sigma(g,h)+\sigma{ (g,\inv{g})}+\sigma{ (1,1)}$
  \end{enumerate}
 \end{enumerate}
\end{lemma}
\begin{proof}
 Denote by $i.x$ the property $[x=1\Rightarrow\delta^2\sigma(x,y,z)=0,\ \forall y,z]$, and so on, so that the property $i$ is the conjunction of the properties:
 $i=i.x\wedge i.y\wedge i.z\wedge i.xy\wedge i.yz\wedge i.xyz$.
 The proof is divided in eight steps.
 
 I) $i\Rightarrow ii$. This is trivial, since $ii=i.xy\wedge i.xyz$.
 
 II) $i.xyz=iii.a$. Indeed, $iii.a=[\delta^2\sigma(g,h,\inv{h}\inv{g})=0,\ \forall g,h]=i.xyz$.
 
 III) $i.xyz\Rightarrow c$.
 Observe that $\inv{g}\cdot g\cdot 1=1\cdot g\cdot\inv{g}=1,\ \forall g$. Thus, $i.xyz\Rightarrow[\sigma(g,1)-\sigma(1,1)=\delta^2\sigma(\inv{g},g,1)=0]\wedge[\sigma(1,1)-\sigma(1,g)=\delta^2\sigma(1,g,\inv{g})=0]$.
 
 IV) $ii\iff iii$. Write down
 \[\delta^2\sigma(h,\inv{h},\inv{g})=\sigma(\inv{h},\inv{g})-\sigma(1,\inv{g})+\sigma(h,\inv{h}\inv{g})-\sigma(h,\inv{h})\ .\]
 Then, by $c$ one has $\sigma(1,\inv{g})=\sigma(1,1)$, so that $i.xy\wedge c\Rightarrow iii.b$, as well as $iii.b\wedge c\Rightarrow i.xy$.
 Therefore, the result follows by II and III.
 
 V) $i.xyz\Rightarrow i.x\wedge i.y\wedge i.z$. It is easily seen that $c\Rightarrow i.x\wedge i.y\wedge i.z$, simply by consideration of $\delta^2\sigma(1,y,z)$, $\delta^2\sigma(x,1,z)$, and $\delta^2\sigma(x,y,1)$. Therefore, this statement follows by III.

 VI) $ii\Rightarrow d$. Firstly, $i.xy\Rightarrow[\delta^2(g,\inv{g},g)=0,\ \forall g]$.
 Thus, writing down explicitly $0=\delta^2\sigma(g,\inv{g},g)=\sigma(\inv{g},g)-\sigma(1,g)+\sigma(g,1)-\sigma(g,\inv{g})$,
 the claim follows by III.
  
 VII) $iii\Rightarrow i$. In view of IV and V, it is left to prove that $iii\Rightarrow i.yz$.
 Observe that $i.yz=[\delta^2\sigma(x,y,\inv{y})=0,\ \forall x,y]$, and explicitly write down
\begin{equation}\label{eq:pf characterization precocycles 1}
 \delta^2\sigma(x,y,\inv{y})=\sigma(y,\inv{y})-\sigma(xy,\inv{y})+\sigma(x,1)-\sigma(x,y)\ .
\end{equation}
 Setting $h=xy$ and $g=\inv{x}$ in $iii.b$, yields
 \[\sigma(\inv{y}\inv{x},x)=-\sigma(xy,\inv{y})+\sigma(xy,\inv{y}\inv{x})+\sigma(1,1)\ .\]
 By the use of $c$, which is allowed by III, $\sigma(1,1)=\sigma(x,1)$, so that
\begin{equation}\label{eq:pf characterization precocycles 2}
-\sigma(xy,\inv{y})+\sigma(x,1)=\sigma(\inv{y}\inv{x},x)-\sigma(xy,\inv{y}\inv{x})\ .\end{equation}
 Therefore, substituting \eqref{eq:pf characterization precocycles 2} in \eqref{eq:pf characterization precocycles 1}, one obtains
\begin{equation}\label{eq:pf characterization precocycles 3}
 \delta^2\sigma(x,y,\inv{y})=\sigma(y,\inv{y})+\sigma(\inv{y}\inv{x},x)-\sigma(xy,\inv{y}\inv{x})-\sigma(x,y)\ .
\end{equation}
 Now, setting $g=\inv{y}\inv{x}$ and $h=x$ in $iii.a$,
 yields
\[\sigma(x,y)=\sigma(\inv{y}\inv{x},x)+\sigma(\inv{y},y)-\sigma(\inv{y}\inv{x},xy)\ ,\]
that is
\begin{equation}\label{eq:pf characterization precocycles 4}
\sigma(\inv{y},y)+\sigma(\inv{y}\inv{x},x)-\sigma(\inv{y}\inv{x},xy)-\sigma(x,y)=0\ .
\end{equation}
The use of $d$, which is allowed by VI, yields $\sigma(\inv{y},y)=\sigma(y,\inv{y})$, and $\sigma(xy,\inv{y}\inv{x})=\sigma(\inv{y}\inv{x},xy)$. Thus, \eqref{eq:pf characterization precocycles 4} implies that \eqref{eq:pf characterization precocycles 2} vanishes, that is $\delta^2\sigma(x,y,\inv{y})=0$.
 
 VIII) $i\Rightarrow e\wedge f\wedge g$.
 Writing down the equalities $\delta^2\sigma(\inv{h}\inv{g},g,h)=0$, $\delta^2\sigma(g,h,\inv{h})=0$, and $\delta^2\sigma(\inv{g},g,h)=0$ explicitly, which are guaranteed by $i.xyz$, $i.yz$, and $i.xy$ respectively, it can be easily seen that the result just requires the use of $c$ and $d$, which is allowed by III and VI.
\end{proof}
The above description of the pre-cocycles allows one to relate these functions with the action of $S_3$ defined in \eqref{Eq:S3 action}.
\begin{proposition}\label{Prop:precohomology as functions}
 Let $G$ be any group. Introduce on $G\times G$ the following equivalence relation:
 for any element $g$ of $G$ set 
 \[(g,1)\sim (1,1)\sim (1,g)\ \ ,\ \ (g,\inv{g})\sim (\inv{g},g)\]
 and, for any pair of elements $g$ and $h$ of $G$ satisfying $g\neq 1\neq h\neq\inv{g}$ set
 \[(g,h)\sim(h,\inv{h}\inv{g})\sim(\inv{h},\inv{g})\ .\]
 Then the group of pre-cocycles $pZ^2(G,A)$ is isomorphic with the group of functions $A^\Omega$ over the set of equivalence classes $\Omega=G\times G/\sim$.
\end{proposition}
\begin{proof}
 Fix a choice of representatives $\Delta:\Omega\to G\times G$ satisfying, for convenience, $\Delta(1,1)_\sim=(1,1)$.
 As for the homomorphism $pZ^2(G,A)\to A^\Omega$ simply take the composition $\sigma\mapsto\sigma\circ\Delta$.
 Whereas, the homomorphism $A^\Omega\to pZ^2(G,A)$, such that $f\mapsto\sigma$, is defined as follows.
First, set
\[\sigma(g,1)=\sigma(1,g)=f(1,1)_\sim\ \ \ ,\ \ \sigma(g,\inv{g})=\sigma(\inv{g},g)=f(g,\inv{g})_\sim\]
for any $g\in G$.
Then, for any other class $\omega$ of $\Omega$, that is $\Delta(\omega)=(g,h)$ where $g\neq 1\neq h\neq\inv{g}$, set
\[\sigma(g,h)=f(\omega)\ ,\]
and, in accordance with the relations $iii.a$, $iii.b$ $iii.e$, $iii.f$ $iii.g$ of \autoref{Lem:characterization of precocycles}, define the values of $\sigma$ over the remaining elements of the orbit as follows:
\[\def\arraystretch{1.5}\begin{array}{l}
 \sigma{ (h,\inv{h}\inv{g})}=f(\omega)+\sigma{ (gh,\inv{h}\inv{g})}-\sigma{ (g,\inv{g})}\\
 \sigma{ (\inv{h},\inv{g})}=-f(\omega)+\sigma{ (g,\inv{g})}+\sigma{ (h,\inv{h})}-\sigma{ (gh,\inv{h}\inv{g})}+\sigma{ (1,1)}\\
 \sigma{ (\inv{h}\inv{g},g)}=f(\omega)+\sigma{ (gh,\inv{h}\inv{g})}-\sigma{ (h,\inv{h})}\\
 \sigma{ (gh,\inv{h})}=-f(\omega)+\sigma{ (h,\inv{h})}+\sigma{ (1,1)}\\
 \sigma{ (\inv{g},gh)}=-f(\omega)+\sigma{ (g,\inv{g})}+\sigma{ (1,1)}\ .
\end{array}\]
It is easy to check that these homomorphisms are inverse of each other.
\end{proof}
In the case of a finite group $G$, the previous result proves that $pZ^2(G,\Z)$ is a finitely generated abelian group of functions, so that standard techniques can be used (see \cite[\S 97]{MR0349869}).
\begin{theorem}\label{Thm:precohomology group}
 Let $G$ be a group of order $n<\infty$, and denote by $G_{(k)}$ the set of elements of order $k$ in $G$.
 Then, for any abelian group $A$, \[pH^2(G,A)\simeq A^{m-n}\oplus A\otimes G/[G,G]\ ,\]
 where $m={(n^2+2|G_{(3)}|+3|G_{(2)}|+5)}/{6}$.
\end{theorem}
\begin{proof}
Since $G$ is finite then $\Omega$ is also finite and, consequently, the group $A^{\Omega}$ is a finitely generated $A$-module.
Denoting $m=|\Omega|$, the above formula for $m$ in terms of $n$, $|G_{(2)}|$, and $|G_{(3)}|$ can be easily achieved (a similar computation is given in \cite{MR3492987}).
Furthermore, if $\Omega=\{\omega_1,\ldots,\omega_m\}$ then
\[A^\Omega=Ah_{\omega_1}\oplus\cdots\oplus Ah_{\omega_m}\ ,\]
where $h_{\omega_i}$ denotes the characteristic function $h_{\omega_{i}}:\omega_j\mapsto\delta_{ij}$ (the Kronecker symbol).
Hence,
\[A^\Omega\simeq A\otimes\Z^\Omega\ \ ,\ \ \Z^\Omega=\Z{h_{\omega_1}}\oplus\cdots\oplus\Z{h_{\omega_m}}\ .\]
By \autoref{Prop:precohomology as functions}, unlike the analogue case in the classical theory, it follows that
\[pZ^2(G,A)\simeq A\otimes pZ^2(G,\Z)\ .\]

In view of this fact, one first considers the case $A=\Z$.
Hence, in $pZ^2(G,\Z)$ the subgroup $Z^2(G,\Z)$ admits a complement:
indeed, given any $\sigma\in\Z^{G\times G}$, and a positive integer $q$, one has $\delta^2(q\sigma)=0$ if and only if $\delta^2\sigma=0$, so that $Z^2(G,\Z)$ is a pure subgroup and thus, since $pZ^2(G,\Z)$ is finitely generated, it is complemented \cite[Theorem 28.2]{MR0255673}.
Hence, denoting by $r$ the free-abelian rank of $Z^2(G,\Z)$, one has $pZ^2(G,\Z)=\Z^{m-r}\oplus Z^2(G,\Z),$
and so \[pH^2(G,\Z)=\Z^{m-r}\oplus H^2(G,\Z)\ .\]
It is well known that, since $G$ is finite, $H^2(G,\Z)\simeq H^1(G,\Q/\Z)\simeq G/[G,G]$ (see \cite[Proposition 6.1, Corollary 10.2]{MR672956}). Therefore, in order to prove the theorem's claim for $A=\Z$ it is left to show that $r=n$.
To see this, it is sufficient to find a finite index subgroup of $Z^2(G,\Z)$ isomorphic with $\Z^n$, and this goal is achieved by taking $B^2(G,\Z)$. Indeed, since $|H^2(G,\Z)|<\infty$, it has finite index and, since $n<\infty$ implies that $H^1(G,\Z)=0$, the desired isomorphism is given by the differential $\delta^1:C_1(G,\Z)=\Z^G\simeq \Z^n\to B^2(G,\Z)$.

Returning to the case of generic coefficients one has to consider the quotient $pZ^2(G,A)/B^2(G,A)$ and, more precisely, to describe $B^2(G,A)$ as a subgroup of $pZ^2(G,A)$.
The only difficulty is that, denoting by $h_g$ the characteristic function relative to $g$ in $C^1(G,\Z)=\Z^G$,
the canonical (finite) set $\{\delta^1h_g\ |\ g\in G\}$ freely generates $B^2(G,\Z)$ as a $\Z$-module,
but in $B^2(G,A)$ the corresponding generating set $\{1_A\otimes\delta^1h_g\ |\ g\in G\}$ is not necessarily free.
Consequently, a more delicate analysis of $pZ^2(G,\Z)$ is needed in order to find a basis presenting a good behavior.
To this aim, one considers the inclusion $B^2(G,\Z)\subseteq Z^2(G,\Z)$ of finitely generated free $\Z$-modules, and uses Smith's normalization theorem to obtain simultaneous basis (see \cite[Theorem 8.61]{MR3443588}); therefore, there exists elements $f_1,\ldots,f_n$ of $Z^2(G,\Z)$ together with non-negative integers $\lambda_1,\ldots,\lambda_n$, which are uniquely determined by imposing the condition that $\lambda_i$ divides $\lambda_{i+1}$ for every $i<n$, satisfying
\[Z^2(G,\Z)=\Z f_1\oplus\cdots\oplus\Z f_n\ \ ,\ \ B^2(G,\Z)=\lambda_1\Z f_1\oplus\cdots\oplus\lambda_n\Z f_n\ .\]
As already remarked, there is an isomorphism in cohomology which allows to compute the integral coefficients as follows:
\[\Z_{\lambda_{1}}\oplus\cdots\oplus\Z_{\lambda_{n}}=H^2(G,\Z)\simeq G/[G,G]\ .\]
At this moment one chooses a basis $\{g_{m+1},\ldots,g_n\}$ of the complement $\Z^{m-n}$ to obtain a decomposition
$pZ^2(G,\Z)=\Z f_1\oplus\cdots\oplus\Z f_n\oplus\Z g_{n+1}\oplus\cdots\oplus\Z g_{m}$.
It is readily seen that $1_A\otimes f_1,\ldots,1_A\otimes g_m$ freely generates $pZ^2(G,A)$ as an $A$-module.
Therefore, denoting by $T=\lambda_1 Af_1\oplus\cdots\oplus \lambda_n A f_n$, it is left to prove that $B^2(G,A)=T$.
Since each $\delta^1h_g\in B^2(G,\Z)$ can be expressed as an integral combination of the elements $\lambda_1f_1,\ldots,\lambda_nf_n$, it follows that $1_A\otimes\delta^1h_g\in T$ for every $g\in G$, so that $B^2(G,A)\leq T$. On the other hand, since the elements $\{\delta^1h_g\ |\ g\in G\}$ constitute a basis of $B^2(G,\Z)$, then every $\lambda_if_i$ is an integral combination of these elements, proving that $\lambda_i1_A\otimes f_i\in B^2(G,A)$, so that $T\leq B^2(G,A)$.
\end{proof}

\section{Partial cohomology relative to an ideal}\label{Section:partial cohomology}
In this section the results about the pre-cohomology groups are generalized to the partial cohomology groups relative to ideals $I\in\Lambda$.

In the above notation, let $\Delta:\Omega\to G\times G$ denote a choice of representatives satisfying $\Delta(1,1)_\sim=(1,1)$,
and let $\Omega^\ast=\Omega\setminus\{(1,1)_\sim\}$.
Then the map $\Pi:G\times G\to\Exel{G}$ defined in \eqref{Eq:product of two generators} induces a map
\[\Pi_\sim:\Omega^\ast\to\Lambda\ ,\ \omega\mapsto\gen{\Pi\circ\Delta(\omega)}_{\mathcal{S}}\cup N_3\ ,\]
which, in view of the equalities \eqref{Eq:generation of ideals inverses} and \eqref{Eq:generation of ideals}, does not depend on the choice of $\Delta$.
In this regard the relative pre-cocycles are characterized as follows:
\begin{lemma}\label{Lem:characterization of precocycles and ideals}
A function $\sigma$ in $A_\infty^{G\times G}$ lies in $Z^2(G,I;A)$ if and only if
 it satisfies
 \begin{enumerate}[i.]
  \item $\Pi(g,h)\in I\iff\sigma(g,h)=\infty$,
  \item $\Pi(x,y,z)\notin I\Rightarrow\delta^2\sigma(x,y,z)=0$.
 \end{enumerate}
Moreover, the group $Z^2(G,I;A)$ is isomorphic with the group of function $A^{\Omega_I}$
where $\Omega_I=\{(1,1)_\sim\}\cup\{\omega\in\Omega^\ast\ |\ \Pi_\sim(\omega)\not\subseteq I\}$.
\end{lemma}
\begin{proof}
 Clearly, any function of the form $\sigma=\sigma'+\epsilon_I$ where $\sigma'\in pZ^2(G,A)$ satisfies $i$.
To see that $\sigma$ also satisfies $ii$, firstly one has to show that no summand of $\delta^2\sigma$ is equal to $\infty$.
Equivalently, that no-one among $\Pi(x,y)$, $\Pi(xy,z)$, $\Pi(x,yz)$, and $\Pi(y,z)$ is contained in $I$.
To this aim, the cases of $\Pi(x,y)$ and $\Pi(y,z)$ are trivial; on the other hand, assuming that $\Pi(xy,z)\in I$ yields $\Pi(\inv{x},x,y,z)=\Pi(\inv{x},xy,z)\in I$, but then also $\Pi(x,y,z)=\Pi(x,\inv{x},x,y,z)\in I$ contradicting the hypothesis, and the case of $\Pi(x,yz)$ is similar.
Therefore, when $\Pi(x,y,z)\notin I$ then $\delta\sigma(x,y,z)=\delta\sigma'(x,y,z)$. However, since $I$ contains $N_3$, $\Pi(x,y,z)\notin I$ implies that $1\in\{x,xy,xyz,yz,z\}$, so that $\delta\sigma'(x,y,z)=0$.

 Vice versa, given $\sigma$ satisfying the properties $i$ and $ii$, one has to produce a pre-cocycle $\sigma'$ for which $\sigma=\sigma'+\epsilon_I$.
 Accordingly to \autoref{Prop:precohomology as functions}, one defines $f$ in $A^{\Omega}$ by setting $f(\omega)=0$ if $\sigma\circ\Delta(\omega)=\infty$, and $f(\omega)=\sigma\circ\Delta(\omega)$ otherwise. Then, the desired $\sigma'$ is obtained by means of the isomorphism $A^\Omega\to pZ^2(G,A)$ as $f\mapsto\sigma'$. 

Finally, the above defined function $f\in A^\Omega$ can be regarded as an element of $A^{\Omega_I}$ simply by restriction. Conversely, any function in $A^{\Omega_I}$ can be extended to an element $f$ of $A^{\Omega}$ by imposing the value $0$ over the complementary set $\Omega\setminus\Omega_I$, resulting as above in an element $\sigma\in Z^2(G,I;A)$.
\end{proof}
For the general case of partial cohomology groups relative to ideals, one is able to come at a weaker form of \autoref{Thm:precohomology group}: 
\begin{theorem}\label{Thm:precohomology group and ideals}
 Let $G$ be a finite group. Then
\[H^2(G,I;\Z)\simeq \Z^{m_I-n_I}\oplus \Z/\mu_1 \Z\oplus\cdots\oplus \Z/\mu_{n_I} \Z\]
 where $m_I=|\Omega_I|\geq n_I=\rk_\Z B^2(G,I;\Z)$, and $\mu_1,\ldots,\mu_{n_I}$ are positive integers ordered by recursive division.
Moreover, \[H^2(G,I;A)\simeq A\otimes H^2(G,I;\Z)\simeq A^{m_I-n_I}\oplus A/\mu_1 A\oplus\cdots\oplus A/\mu_{n_I} A\] for any abelian group $A$.
\end{theorem}
\begin{proof}
The result follows in a way similar to the proof of \autoref{Thm:precohomology group}, with the following minor modification.
In this case, using \autoref{Lem:characterization of precocycles and ideals}, one also has $Z^2(G,I;A)\simeq A\otimes Z^2(G,I;\Z)$.
Then, applying Smith's normalization theorem, one obtains a free basis $f_1,\ldots,f_{m_I}$ of $Z^2(G,I;\Z)$ together with positive integers $\mu_1,\ldots,\mu_{n_I}$ such that $\mu_1f_1,\ldots,\mu_{n_I}f_{n_I}$ is a free basis of $B^2(G,I;\Z)$.
Finally, the fact $B^2(G,I;A)=\mu_1Af_1\oplus\cdots\oplus\mu_{n_I}Af_{n_i}$ is proved considering the canonical generating set $\{\delta h_g+\epsilon_I\ |\ g\in G\}$.
\end{proof}
 In particular, when $A$ is a divisible group, then for every ideal $I$ in $\Lambda$ each summand $A/\mu_i A$ in the decomposition of $H^2(G,I;A)$ is trivial. This observation confirms the conjecture mentioned in the introduction:
\begin{corollary}\label{Cor:partial multiplier}
 If $G$ is finite and $\K$ is an algebraically closed field, then any component of the partial multiplier is a finite direct sum of copies of $\K^\times$.
\end{corollary}

 \subsection{An example}
 It can be seen in the proof of \autoref{Thm:precohomology group} that the torsion part of $pH^2(G,\Z)$ comes from the classical cohomology group $H^2(G,\Z)$.
 A similar statement can not be formulated for a generic ideal $I$.
 Moreover, the next example shows that the torsion part of $H^2(G,I;\Z)$ is not necessarily a quotient of $H^2(G,\Z)$.
 The general procedure to compute $H^2(G,I;\Z)$ explicitly, which can be reproduced for various small groups using \texttt{GAP} \cite{GAP4}, is summarized as follows:
 one determines $m=|\Omega|$ together with a set of representatives $\Delta(\omega_1),\ldots,\Delta(\omega_m)$ for the classes of $\Omega$;
 the matrix \[M=(\delta^1h_{g_i}\circ\Delta(\omega_j))_{i,j}\in\mathrm{Mat}_{n\times m}(\Z)\]
 yields the coefficients of the natural basis $\delta^1h_{g_1},\ldots,\delta^1h_{g_n}$ of $B^2(G,\Z)$ in terms of the characteristic functions $h_{\omega_1},\ldots,h_{\omega_m}$;
given an ideal $I$ in $\Lambda$, one determines $\Omega\setminus\Omega_I=\{\omega_{i_1},\ldots,\omega_{i_{m-m_I}}\}$ to obtain, removing from $A$ the columns $i_1,\ldots,i_{m-m_I}$, a matrix $M_I$;
finally, the coefficients $\mu_1,\ldots,\mu_{n_I}$ which describe the torsion part of $H^2(G,I;\Z)$ in \autoref{Thm:precohomology group and ideals} are the non-zero diagonal entries of the Smith's normal form of the matrix $M_I$.
\begin{example}
\begin{figure}
\[\xymatrix@C=.3pc@R=3.5pc{
&&N_3:\Z\oplus \Z\oplus \Z_6\ar@{-}[d]\ar@{-}[dl]\ar@{-}[dr]\ar@{-}[drr]
&&\\
&\omega_5:\Z\oplus \Z_6\ar@{-}[d]\ar@{-}[dl]\ar@{-}[dr]
&\omega_6:\Z\oplus \Z_6\ar@{-}[dll]\ar@{-}[dr]\ar@{-}[drr]
&\omega_7:\Z\oplus \Z_6\ar@{-}[dll]\ar@{-}[d]\ar@{-}[drr]
&\omega_8:\Z\oplus \Z_6\ar@{-}[dll]\ar@{-}[d]\ar@{-}[dr]
&\\
\Z_6\ar@{-}[d]\ar@{-}[dr]
&\Z_6\ar@{-}[dl]\ar@{-}[dr]
&\Z\oplus\Z_2\ar@{-}[dl]\ar@{-}[d]
&\Z_2\oplus\Z_6\ar@{-}[dlll]\ar@{-}[d]\ar@{-}[dr]
&\Z_6\ar@{-}[dlll]\ar@{-}[d]
&\Z_6\ar@{-}[dlll]\ar@{-}[dl]
\\
\Z_6\ar@{-}[d]\ar@{-}[dr]\ar@{-}[drr]
&\Z_2\ar@{-}[d]
&\Z_2\ar@{-}[dl]
&\omega_4:\Z_6\ar@{-}[dl]\ar@{-}[dr]
&\Z_2\ar@{-}[d]\ar@{-}[dlll]
&\\
\omega_2:\Z_6\ar@{-}[d]\ar@{-}[drr]
&\Z_2\ar@{-}[d]
&\Z_3\ar@{-}[d]
&&0\ar@{-}[d]
&\\
\Z_2\ar@{-}[dr]
&\omega_3:\Z_2\ar@{-}[d]\ar@{-}[drr]
&\Z_3\ar@{-}[dd]
&&0\ar@{-}[dl]
\\
&\Z_2\ar@{-}[dr]
&&0\ar@{-}[dl]
\\
&&0
}\]
\caption{\label{Fig:lattice for Z6}
The partial cohomology semilattice of groups
$H^2(\Z_6,\Lambda;\Z)$.
The semilattice structure is the same of $(\Lambda,\cup)$.
At the vertex corresponding to the ideal $I$ appears $H^2(\Z_6,I;\Z)$.
The principal-modulo-$N_3$ ideals are marked by their generator, thus ``$\omega_5:\Z\oplus\Z_6$'' stands for
$H^2(G,\Pi_\sim(\omega_5);\Z)\simeq\Z\oplus\Z_6$.}
\end{figure}
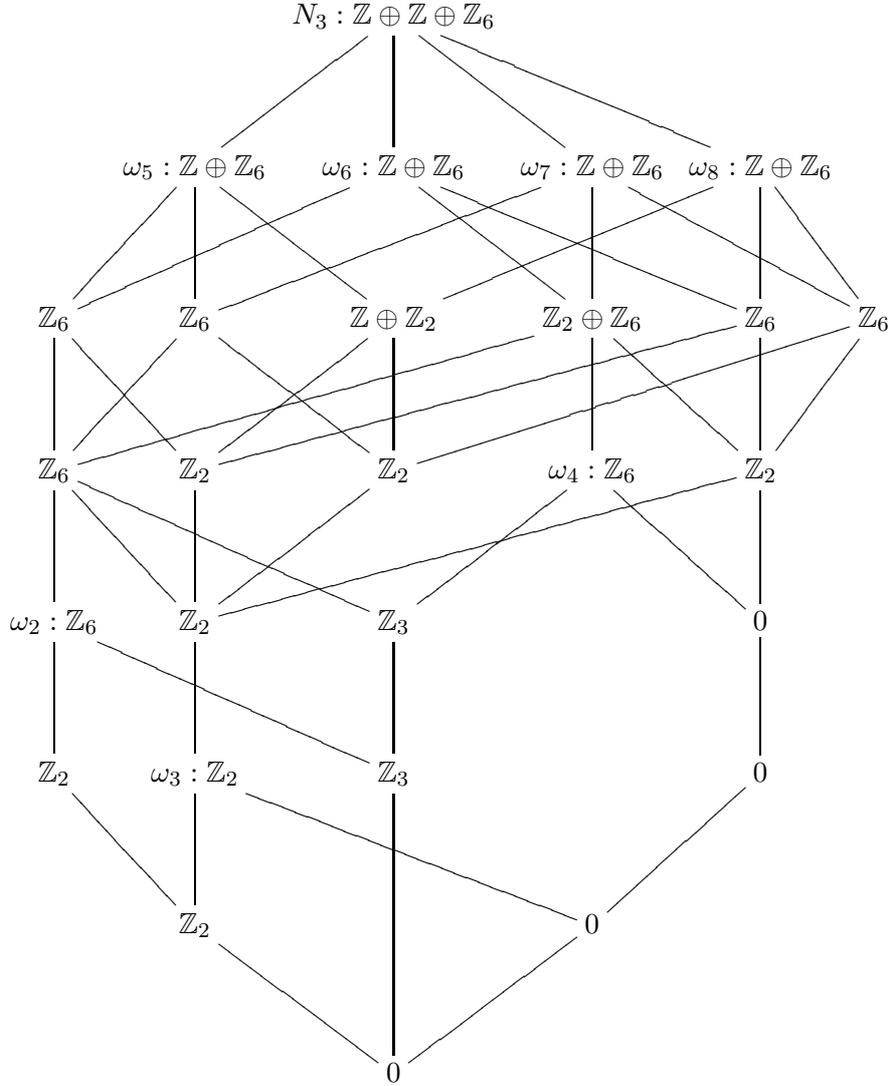
 Consider the cyclic group $G=\Z_6$. It is directly seen that $G$ admits $8=(6^2+2\cdot 2+3\cdot 1+5)/6$ classes in $\Omega$, which are represented by the following pairs:
  \[\begin{array}{r|cccccccc}
     &\omega_1&\omega_2&\omega_3&\omega_4&\omega_5&\omega_6&\omega_7&\omega_8\\\hline
     \Delta&(0,0)&(1,5)&(2,4)&(3,3)&(1,1)&(1,2)&(1,3)&(2,2)
    \end{array}\]
    Thus, the coefficients of the matrix $M=(\delta^1h_{g_i}\circ\Delta(\omega_j))_{i,j}\in\mathrm{Mat}_{6\times 8}(\Z)$ can be read in the following table:
  \[\begin{array}{r|rrrrrrrr}
     &\omega_1&\omega_2&\omega_3&\omega_4&\omega_5&\omega_6&\omega_7&\omega_8\\\hline
     \delta h_0&1&-1&-1&-1&\cdot&\cdot&\cdot&\cdot\\
     \delta h_1&\cdot&1&\cdot&\cdot&2&1&1&\cdot\\
     \delta h_2&\cdot&\cdot&1&\cdot&-1&1&\cdot&2\\
     \delta h_3&\cdot&\cdot&\cdot&2&\cdot&-1&1&\cdot\\
     \delta h_4&\cdot&\cdot&1&\cdot&\cdot&\cdot&-1&-1\\
     \delta h_5&\cdot&1&\cdot&\cdot&\cdot&\cdot&\cdot&\cdot
    \end{array}\]
It can be checked that the Smith's normalization algorithm provides two matrices
$P\in\GL(6,\Z)$ and $Q\in\GL(8,\Z)$, such that the diagonal entries of $D=PMQ$ are $\lambda_1=\cdots=\lambda_5=1$ and $\lambda_6=6$.
In accordance with \autoref{Thm:precohomology group}, then $pH^2(\Z_6,\Z)\simeq\Z\oplus\Z\oplus\Z_6$.

Consider now $I=N_3\cup\geni{\Pi(1,2),\Pi(1,3)}$, so that $\Omega\setminus\Omega_I=\{\omega_6,\omega_7\}$.
In particular $n_I=|\Omega_I|=6$, and Smith's normalization for $M_I$ provides a matrix with diagonal entries $\mu_1=\ldots=\mu_4=1$, $\mu_5=2$ and $\mu_6=6$. Therefore, $H^2(\Z_6,I;\Z)\simeq\Z_2\oplus\Z_6$ that, as anticipated, is not a quotient of $H^2(\Z_6,\Z)=\Z_6$.

For another instance, consider $I=N_3\cup\geni{\Pi(2,4),\Pi(3,3)}$.
Since, by \eqref{Eq:generation of ideals inverses}, $\geni{\Pi(2,4)}=\geni{\Pi(2)}=\geni{\Pi(4)}$ and $\geni{\Pi(3,3)}=\geni{\Pi(3)}$,
it immediately follows that $\Omega\setminus\Omega_I=\{\omega_3,\ldots,\omega_8\}$.
Thus, only the first and the second column in the above table are relevant, and one sees that $\delta^1h_5+\epsilon_I=h_{\omega_2}+\epsilon_I$ and $\delta^1(h_0+h_5)+\epsilon_I=h_{\omega_1}+\epsilon_I$.
Consequently, $H^2(G,I;\Z)=0$ in this case.

The full description of $H^2(\Z_6,\Lambda;\Z)$ is shown in \autoref{Fig:lattice for Z6}.
\end{example}

\section{Consistency with the previous notion}\label{Section:coherence}
It is left to show that for any abelian group of coefficients $A$ and any ideal $I$ in $\Lambda$, the group $H^2(G,I;A)$ is a partial cohomology group according to \cite{MR3312299}.
To this aim, one has to associate the pair $(I;A)$ with a unital partial $G$-module $B$ (that is to say, a commutative monoid together with a unital partial $G$-action).
\begin{theorem}\label{Thm:coherence with the previous notion}
For any pair $(I;A)$, there exists a unital partial $G$-module $B$ such that $H^2(G,I;A)\simeq H^2(G,B)$.
\end{theorem}
\begin{proof}
 The construction is similar to that of \autoref{Lem:extension of E3 by a precocycle}:
 let $E(\Exel{G}\setminus I)$ denote the collection of the idempotents in $\Exel{G}\setminus I$, then the set
 \[B=\left(A\times E(\Exel{G}\setminus I)\right)\cup\{\infty\}\] endowed with zero element $\infty$ and multiplication
\[(a,\varepsilon)\cdot(b,\eta)=\left\{\begin{array}{cl}
        (a+b,\varepsilon\eta)&\mbox{ if }\varepsilon\eta\notin I\\
        \infty&\mbox{ otherwise,}
        \end{array}\right.\]
 results in an $A$-cancellative central extension of the semilattice $E(\Exel{G}/I)$.
 Moreover, $B$ is a commutative monoid with identity element $(0_A,\Pi(1))$, and it is naturally endowed with the following unital partial action:
 for any $g\in G$, one considers the idempotent
 \[e_g =\left\{\begin{array}{cl}
        (0_A,\Pi(g,\inv{g}))&\mbox{ if }\Pi(g)\notin I\\
        \infty&\mbox{ otherwise}
        \end{array}\right.\]
together with the ideal $B_g=Be_g$ of $B$, and the partial action $\vartheta$ is given by the family of isomorphisms $\{\vartheta_g:B_{\inv{g}}\to B_g\ |\ g\in G\}$, defined by:
 \[\vartheta_g(a,\varepsilon)
 =\left\{\begin{array}{cl}
        (a,\Pi(g)\varepsilon\Pi(\inv{g}))&\mbox{ if }\Pi(g)\varepsilon\Pi(\inv{g})\notin I\\
        \infty&\mbox{ otherwise.}
        \end{array}\right.\]
Thus, for any $x_1,\ldots,x_n\in G$, the ideal $B_{(x_1,\ldots,x_n)}=B_{x_1}B_{x_1x_2}\cdots B_{x_1x_2\cdots x_n}$ is either reduced to the trivial semigroup $\{\infty\}$, or it is generated by the idempotent
\[e_{x_1}e_{x_1x_2}\cdots e_{x_1\cdots x_n}=\left(0_A,\Pi(x_1,x_2,\ldots,x_n,\inv{(x_1x_2\cdots x_n)})\right)\ ,\]
which obviously plays the role of the identity element.
In the latter case, the subgroup of $B_{(x_1,\ldots,x_n)}$ consisting of invertible elements is
\[\mathcal{U}(B_{(x_1,\ldots,x_n)})=\left\{\left(a,\Pi(x_1,x_2,\ldots,x_n,\inv{(x_1x_2\cdots x_n)})\right)\ |\ a\in A\right\}\ ,\]
which is isomorphic to $A$ by means of the projection on the first component $\pi:B\setminus\{\infty\}\to A$.
It is convenient to extend $\pi$ to the whole $B$ by the assignment of $\pi(\infty)=\infty$.
Therefore, an $n$-cochain $\sigma\in C^n(G,B)$, which is by definition a map $\sigma:G^n\to B$ satisfying
$\sigma(x_1,\ldots,x_n)\in\mathcal{U}(B_{(x_1,\ldots,x_n)})$ for every $x_1,\ldots,x_n\in G$, can be identified with the composite function $\bar\sigma=\pi\circ\sigma:G^n\to A_\infty$.
Clearly, one has that $\bar\sigma(x_1,\ldots,x_n)=\infty$ if and only if $\Pi(x_1,\ldots,x_n,\inv{(x_1\ldots x_n)})\in I$, which indeed corresponds to the case $B_{(x_1,\ldots,x_n)}=\{\infty\}$.
By the above construction of the partial action of $G$ over $B$, the coboundary map $\delta^n:C^n(G,B)\to C^{n+1}(G,B)$ is given by the relation
\[\delta^n\sigma(x_1,\ldots,x_{n+1})=\left\{\begin{array}{cl}
        (\delta^n\bar\sigma(x_1,\ldots,x_{n+1}),\eta)&\mbox{ if }\eta\notin I\\
        \infty&\mbox{ otherwise,}
        \end{array}\right.\]
where \[\eta=\Pi(x_1,\ldots,x_{n+1},\inv{(x_1\cdots x_{n+1})})\] and $\delta^n:C^n(G,A)\to C^{n+1}(G,A)$ is the standard coboundary homomorphism relative to the the trivial (global) $G$-module $A$.
Notice that, in view of \eqref{Eq:ideals and associativity}, the condition $\eta\notin I$ is equivalent to  $\Pi(x_1,\ldots,x_{n+1})\notin I$.
Therefore, taking $n=2$, \autoref{Lem:characterization of precocycles and ideals} shows that $\sigma\in Z^2(G,B)$ if and only if $\bar\sigma\in Z^2(G,I;A)$.
Finally, one checks the case $n=1$ to conclude that $\sigma\in B^2(G,B)$ if and only if $\bar\sigma\in B^2(G,I;A)$, so that the map $\sigma\to\bar\sigma$ induces the desired isomorphism of cohomology groups.
\end{proof}



\end{document}